\documentclass[reqno]{amsart}

\allowdisplaybreaks[1]

\usepackage{amssymb}
\usepackage{amsmath}
\usepackage{amsthm}
\usepackage[all]{xy}
\usepackage{tikz}
\usepackage{latexsym}
\usepackage[justification=centering]{caption}
\usepackage{hyperref}
\hypersetup{
 colorlinks=true,
 linkcolor=blue,
 citecolor=blue}

 \tikzset{vertex/.style={fill,circle,inner sep=1.0pt}}

\newtheorem{theorem}{Theorem}[section]
\newtheorem{proposition}[theorem]{Proposition}
\newtheorem{corollary}[theorem]{Corollary}
\newtheorem{lemma}[theorem]{Lemma}
\theoremstyle{definition}
\newtheorem{definition}[theorem]{Definition}
\newtheorem{remark}[theorem]{Remark}
\newtheorem{example}[theorem]{Example}

\newcommand{\lex}[2]{{#1}[{#2}]}
\newcommand{\neib}[2]{\overline{N}_{#1} (#2)}

\newcommand{\sphere}[1]{S^{#1}}

\title[]{Independence complex of the lexicographic product of a forest}
\author[]{Kengo Okura}
\address{Osaka Prefecture University, 1-1 Gakuen-cho, Naka-ku, Sakai, Osaka 599-8531, Japan}
\email{okura.kengo.k35@kyoto-u.jp}

\begin{document}

\begin{abstract}
We study the independence complex of the lexicographic product $\lex{G}{H}$ of a forest $G$ and a graph $H$. 
We prove that for a forest $G$ which is not dominated by a single vertex, if the independence complex of $H$ is homotopy equivalent to a wedge sum of spheres, then so is the independence complex of $\lex{G}{H}$. 
We offer two examples of explicit calculations. As the first example, we determine the homotopy type of the independence complex of $\lex{L_m}{H}$, where $L_m$ is the tree on $m$ vertices with no branches, for any positive integer $m$ when the independence complex of $H$ is homotopy equivalent to a wedge sum of $n$ copies of $d$-dimensional sphere. 
As the second one, for a forest $G$ and a complete graph $K$, we describe the homological connectivity of the independence complex of $\lex{G}{K}$ by the independent domination number of $G$.
\end{abstract}

\maketitle

\section{Introduction}
\label{introduction}
In this paper, a {\it graph} $G$ always means a finite undirected graph with no multiple edges and loops. Its vertex set and edge set are denoted by $V(G)$ and $E(G)$, respectively.
A subset $\sigma$ of $V(G)$ is an {\it independent set} if any two vertices of $\sigma$ are not adjacent. The independent sets of $G$ are closed under taking subset, so they form an abstract simplicial complex. We call this abstract simplicial complex the {\it independence complex} of $G$ and denote by $I(G)$. In the rest of this paper, $I(G)$ denotes a geometric realization of $I(G)$ unless otherwise noted.

Independence complexes of graphs are no less important than other simplicial complexes constructed from graphs and have been studied in many contexts. 
In particular, the independence complexes of square grid graphs are studied by Thapper \cite{Thapper08}, Iriye \cite{Iriye12} and many other researchers. It is conjectured by Iriye \cite[Conjecture 1.8]{Iriye12} that the independence complex of cylindrical square grid graph is always homotopy equivalent to a wedge sum of spheres.
{\it Discrete Morse theory} , introduced by Forman \cite{Forman98} and reformulated by Chari \cite{Chari00}, is one of the effective methods for determining the homotopy type of independence complex. Bousquet-M{\'{e}}lou, Linusson and Nevo \cite{BousquetmelouLinussonNevo08} and Thapper \cite{Thapper08} studied the independence complexes of grid graphs by performing discrete Morse theory as a combinatorial algorithm called {\it matching tree}. However, it is hard to distinguish two complexes which has the same number of cells in each dimension only by discrete Morse theory. This is precisely the situation which we have to deal with in this paper. We need topological approaches in case that discrete Morse theory is not available. For example, it is effective to represent an independence complex of a graph as a union of independence complexes of subgraphs, as in Engstr{\"{o}}m \cite{Engstrom09}, Adamaszek \cite{Adamaszek12} and Barmak \cite{Barmak13}. 

Let $L_m$ be a tree on $m$ vertices with no branches, and $C_n$ be a cycle on $n$ vertices ($n \geq 3$). Namely
\begin{align*}
&V(L_m)=\{1,2,\ldots, m\}, & &E(L_m) = \{ij  \ |\ |i-j|=1 \} , \\
&V(C_n) = \{1,2, \ldots, n \}, & &E(C_n) = E(L_n) \cup \{n1 \}.
\end{align*}
Related to the above previous researches, we focus on the fact that the cylindrical square grid graphs are obtained from $L_m$ and $C_n$ by a certain ``product'' construction. As Harary \cite{Harary69} mentioned, there are various ways to construct a graph structure on $V(G_1) \times V(G_2)$ for given two graphs $G_1$ and $G_2$. A cylindrical square grid graph is the {\it Cartesian product} of $L_m$ and $C_n$ for some $m, n$. In this paper, we are interested in the {\it lexicographic product} of two graphs, which is defined as follows.
\begin{definition}
Let $G, H$ be graphs. The {\it lexicographic product} $\lex{G}{H}$ is a graph defined by
\begin{align*}
&V(\lex{G}{H}) = V(G) \times V(H) ,\\
&E(\lex{G}{H}) = \left\{ (u_1, v_1)(u_2, v_2) \ \middle| \ 
\begin{aligned}
&u_1 u_2 \in E(G) \\
&\text{ or} \\
&u_1=u_2, v_1 v_2 \in E(H) 
\end{aligned}
\right\}.
\end{align*}
\end{definition}
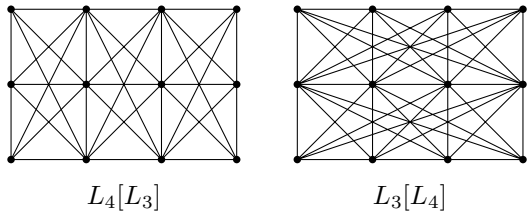
\begin{figure}[htb]
\begin{tabular}{ccc}
\begin{tikzpicture}
\draw (1,1) grid (4,3);
\draw (1,1)--(3,3) (1,2)--(2,3) (2,1)--(4,3) (3,1)--(4,2) (1,3)--(3,1) (1,2)--(2,1) (2,3)--(4,1) (3,3)--(4,2);
\draw (1,1)--(2,3) (1,3)--(2,1) (2,1)--(3,3) (2,3)--(3,1) (3,1)--(4,3) (3,3)--(4,1);
\foreach \x in {1,2,3,4} {\foreach \y in {1,2,3 } {\node at (\x, \y) [vertex] {};};}
\node at (2.5,0.5) {$\lex{L_4}{L_3}$};
\end{tikzpicture}
& &
\begin{tikzpicture}
\draw (1,1) grid (4,3);
\draw (1,1)--(3,3) (1,2)--(2,3) (2,1)--(4,3) (3,1)--(4,2) (1,3)--(3,1) (1,2)--(2,1) (2,3)--(4,1) (3,3)--(4,2);
\draw (1,1)--(3,2) (1,1)--(4,2) (2,1)--(4,2) (3,1)--(1,2) (4,1)--(1,2) (4,1)--(2,2) (1,2)--(3,3) (1,2)--(4,3) (2,2)--(4,3) (3,2)--(1,3) (4,2)--(1,3) (4,2)--(2,3);
\foreach \x in {1,2,3,4} {\foreach \y in {1,2,3 } {\node at (\x, \y) [vertex] {};};}
\node at (2.5,0.5) {$\lex{L_3}{L_4}$};
\end{tikzpicture}
\end{tabular}
\caption{Lexicographic products $\lex{L_4}{L_3}$ and $\lex{L_3}{L_4}$.}
\end{figure}

\noindent
Harary \cite{Harary69} called this construction the {\it composition}. A lexicographic product $\lex{G}{H}$ can be regarded to have $|V(G)|$ pseudo-vertices. Each of them is isomorphic to $H$ and two pseudo-vertices are ``adjacent'' if the corresponding vertices of $G$ are adjacent. Graph invariants of lexicographic product have been investigated by, for example, Geller and Stahl \cite{GellerStahl75}. Independence complexes of lexicographic products are studied by Vander Meulen and Van Tuyl \cite{VandermeulenVantuyl17} from combinatorial point of view.

We try to reveal in what condition the independence complex of a lexicographic product is homotopy equivalent to a wedge sum of spheres. The main result of this paper is the following theorem.
\begin{theorem}
\label{forest}
Let $G$ be a forest and $H$ be a graph. We call $G$ a {\it star} if there exists $v \in V(G)$ such that $uv \in E(G)$ for any $u \in V(G) \setminus \{v\}$.
Suppose that $I(H)$ is homotopy equivalent to a wedge sum of spheres. Then, we have the followings. 
\begin{enumerate}
\item If $G$ is a star on at least $2$ vertices, then $I(\lex{G}{H})$ is homotopy equivalent to a disjoint union of two wedge sums of spheres. 
\item If $G$ is not a star, then $I(\lex{G}{H})$ is homotopy equivalent to a wedge sum of spheres.
\end{enumerate}
\end{theorem}
\noindent
For example, Kozlov \cite[Proposition 5.2]{Kozlov99} proved that $I(C_n)$ is homotopy equivalent to a wedge sum of spheres. So, it follows from Theorem \ref{forest} that $I(\lex{L_m}{C_n})$ with $m \geq 4$ is homotopy equivalent to a wedge sum of spheres. 
Remark that $\lex{L_m}{C_n}$ contains a cylindrical square grid graph as a subgraph which is obtained from $\lex{L_m}{C_n}$ by removing edges.
Furthermore, we determined the homotopy type of $I(\lex{L_m}{H})$ for any $m \geq 1$ and a graph $H$ such that $I(H)$ is homotopy equivalent to $n$ copies of $k$-dimensional spheres. 
We denote the $d$-dimensional sphere by $\sphere{d}$ and a wedge sum of $n$ copies of CW complex $X$ by $\bigvee_{n} X$.
\begin{theorem}
\label{line theorem}
Let $H$ be a graph such that $I(H) \simeq {\bigvee}_n \sphere{k}$ with $n \geq 1$, $k \geq 0$. Then we have
\begin{align*}
&I(\lex{L_m}{H})  \\
\simeq &\left\{
\begin{aligned}
&{\bigvee}_n \sphere{k} & &(m=1), \\
&\left( {\bigvee}_n \sphere{k} \right) \sqcup \left( {\bigvee}_n \sphere{k} \right) & &(m=2), \\
&\left( {\bigvee}_n \sphere{k} \right) \sqcup  \left( {\bigvee}_{n^2} \sphere{2k+1} \right)& &(m=3), \\
&\bigvee_{0 \leq p \leq \frac{m+1}{2}} \left( \bigvee_{pk -1 +\max \left\{p, \frac{m}{3} \right\} \leq d \leq pk+\frac{m+p-2}{3}}
\left( {\bigvee}_{N_{m,n,k}(p,d)} \sphere{d} \right) \right) & &(m \geq 4), 
\end{aligned} \right.
\end{align*}
where
\begin{align*}
N_{m,n,k}(p,d) &= n^p \binom{d-pk+1}{p} \binom{p+1}{3(d-pk+1)-m} .
\end{align*}
\end{theorem}
\noindent
Here, $\binom{l}{r}$ denotes the binomial coefficient. We define $\binom{l}{r}=0$ if $r<0$ or $l <r$. 

The rest of this paper is organized as follows. 
In Section \ref{preliminaries}, we define notations on graphs and state some of the basic properties of independence complexes of graphs. 
Section \ref{proof of main theorem} is the main part of this paper. It first provides a condition for the independence complex of a graph to be the union of the independence complexes of given two full subgraphs (Lemma \ref{ind pushout}). Note that the cofiber sequence studied by Adamaszek \cite[Proposition 3.1]{Adamaszek12} is a special case of this decomposition. Using this result, we obtain a decomposition of an independence complex of a lexicographic product, which is essentially important to achieve our purpose (Theorem \ref{splitting}). Then, we prove Theorem \ref{forest}. Here we need an observation on the unreduced suspension of a disjoint union of two spaces (Lemma \ref{disjoint suspension}). 
Section \ref{explicit calculations} contains two examples of the explicit calculations. The first one is the proof of Theorem \ref{line theorem}. The second one is on the relationship between the homological connectivity of $I(\lex{G}{H})$ and the independent domination number of a forest $G$ (Theorem \ref{connectivity and domination}).

\section{Preliminaries}
\label{preliminaries}
In this paper, a {\it graph} always means a {\it finite undirected simple graph} $G$. It is a pair $(V(G), E(G))$, where $V(G)$ is a finite set and $E(G)$ is a subset of $2^{V(G)}$ such that $|e|=2$ for any $e \in E(G)$. An element of $V(G)$ is called a {\it vertex} of $G$, and an element of $E(G)$ is called an {\it edge} of $G$. In order to indicate that $e=\{u, v\}$ ($u,v \in V(G)$), we write $e =uv$.

For a vertex $v \in V(G)$, an {\it open neighborhood} $N_G (v)$ of $v$ in $G$ is defined by 
\begin{align*}
N_G (v) = \{ u \in V(G) \ |\ uv \in E(G) \}. 
\end{align*}
A {\it closed neighborhood} $\neib{G}{v}$ of $v$ in $G$ is defined by $\neib{G}{v} = N_G (v) \sqcup \{ v\}$.

A {\it full subgraph} $H$ of a graph $G$ is a graph such that 
\begin{align*}
V(H) &\subset V(G), \\
E(H) &=\{ uv \in E(G) \ |\ u, v \in V(H) \}.
\end{align*}
For two full subgraphs $H, K$ of $G$, a full subgraph whose vertex set is $V(H) \cap V(K)$ is denoted by $H \cap K$, and a full subgraph whose vertex set is $V(H) \setminus V(K)$ is denoted by $H \setminus K$. For a subset $U \subset V(G)$, $G \setminus U$ is the full subgraph of $G$ such that $V(G \setminus U) = V(G) \setminus U$.

An {\it abstract simplicial complex} $K$ is a collection of finite subsets of a given set $V(K)$ such that
if $\sigma \in K$ and $\tau \subset \sigma$, then $\tau \in K$. An element of $K$ is called a {\it simplex} of $K$. 
For a simplex $\sigma$ of $K$, we set $\dim \sigma = |\sigma| -1 $, where $|\sigma|$ is the cardinality of $\sigma$. 
As noted in Section \ref{introduction}, we do not distinguish an abstract simplicial complex $K$ from its geometric realization $|K|$.

The {\it independence complex} $I(G)$ of a graph $G$ is an abstract simplicial complex defined by
\begin{align*}
I(G) = \{ \sigma \subset V(G) \ |\ uv \notin E(G) \text{ for any $u, v \in \sigma$ } \}.
\end{align*}

For a full subgraph $H$ of $G$, $I(H)$ is a subcomplex of $I(G)$. Furthermore, if $H, K$ are full subgraphs of $G$, then $I(H \cap K) = I(H) \cap I(K)$.
The following proposition is the fundamental property of independence complexes.
\begin{proposition}
\label{disjoint union and join}
Let $G$ be a graph and $G_1$ and $G_2$ be full subgraphs of $G$ such that $V(G)=V(G_1) \sqcup V(G_2)$. 
\begin{enumerate}
\item If $uv \notin E(G)$ for any $u \in V(G_1)$ and $v \in V(G_2)$, then we have
\begin{align*}
I(G) = I(G_1) * I(G_2).
\end{align*}
\item If $uv \in E(G)$ for any $u \in V(G_1)$ and $v \in V(G_2)$, then we have
\begin{align*}
I(G) = I(G_1) \sqcup I(G_2).
\end{align*}
\end{enumerate}
\end{proposition}
\begin{proof}
In the proof, we consider $I(G)$ as an abstract simplicial complex.

Suppose that $uv \notin E(G)$ for any $u \in V(G_1)$ and $v \in V(G_2)$. Then, we have
\begin{align*}
I(G) &=  \left\{\sigma \subset V(G_1) \sqcup V(G_2) \ \middle|\ \left.
\begin{aligned}
&\sigma \cap V(G_1) \in I(G_1) \\
&\text{ and }\\
&\sigma \cap V(G_2) \in I(G_2) 
\end{aligned} \right. \right\}\\
&= I(G_1) * I(G_2) .
\end{align*}

Suppose that $uv \in E(G)$ for any $u \in V(G_1)$ and $v \in V(G_2)$. Then, we have
\begin{align*}
I(G) &=  \left\{\sigma \subset V(G_1) \sqcup V(G_2) \ \middle|\ \left. 
\begin{aligned}
&\sigma \subset V(G_1) \text{ and } \sigma \in I(G_1) \\
&\text{ or } \\
&\sigma \subset V(G_2) \text{ and } \sigma \in I(G_2)
\end{aligned} \right. \right\} \\
&= I(G_1) \sqcup I(G_2) .
\end{align*}
\end{proof}

Let $X$ be a CW complex. We denote the {\it unreduced} suspension of $X$ by $\Sigma X$. For subcomplexes $X_1, X_2$ of $X$ such that $X_1 \cap X_2 =A$, we denote the union of $X_1$ and $X_2$ by $X_1 \cup_A X_2$ in order to indicate that the intersection of $X_1$ and $X_2$ is $A$.

\section{Proof of Theorem \ref{forest}}
\label{proof of main theorem}
We first prove the following theorem, which we need to prove Theorem \ref{forest}.
\begin{theorem}
\label{splitting}
Let $G$ a graph and $v$ be a vertex of $G$. Suppose that there exists a vertex $w$ of $G$ such that $N_G (w) = \{v\}$. Let $H$ be a non-empty graph.
\begin{itemize}
\item If $G \setminus \neib{G}{v} = \emptyset$, then we have
\begin{align*}
I(\lex{G}{H}) = I(H) \sqcup I(\lex{(G \setminus \{v\})}{H}) .
\end{align*}
\item If $G \setminus \neib{G}{v} \neq \emptyset$, then we have
\begin{align*}
I(\lex{G}{H}) \simeq &\Sigma I(\lex{(G \setminus \neib{G}{v} )}{H}) \vee \left(I(\lex{(G \setminus \neib{G}{v} )}{H}) * I(H) \right) \\
&\ \vee \left(I(\lex{(G \setminus\{v, w\})}{H}) * I(H) \right) .
\end{align*}
\end{itemize}
\end{theorem}
The proof of Theorem \ref{splitting} has two steps. The first step is to decompose $I(\lex{G}{H})$ as a union of $I(\lex{(G \setminus N_G (v))}{H})$ and $I(\lex{(G \setminus \{v\})}{H})$. The second step is to transform this union into a wedge sum. We need two lemmas corresponding to these two steps.
\begin{lemma}
\label{ind pushout}
Let $G$ be a graph and $H, K \subset G$ be full subgraphs of $G$ such that $V(H) \cup V(K) =V(G)$. Suppose that $v_1 v_2 \in E(G)$ for any vertices $v_1 \in V(H) \setminus V(K)$ and $v_2 \in V(K) \setminus V(H)$. Then,
\begin{align*}
I(G) = I(H) \cup_{I(H \cap K)} I(K).
\end{align*}
\end{lemma}
\begin{proof}
For a simplex $\sigma$ of $I(G)$, suppose that there exists a vertex $u_0 \in \sigma \cap (V(H) \setminus V(K))$. Then, by the assumption of the lemma, any vertex $v \in V(K) \setminus V(H)$ is adjacent to $u_0$. So, $\sigma \cap (V(K) \setminus V(H))$ must be empty, which means that $\sigma$ is a simplex of $H$.  On the other hand, if $\sigma \cap (V(H) \setminus V(K)) = \emptyset$, then $\sigma$ is a simplex of $K$ since $V(H) \cup V(K) = V(G)$.
\end{proof}
\begin{figure}[htb]
\begin{tabular}{c}
\begin{tikzpicture}
\foreach \x in {(0,1),(1,1),(2,1),(3,1),(2,0),(3,0),(4,0)} {\node at \x [vertex] {};}
\draw (0,1)--(3,1) (2,0)--(4,0);
\node at (-1,0) {$K$};
\node at (-1,1) {$H$};
\end{tikzpicture}
\\
\\
\begin{tikzpicture}
\foreach \x in {(0,1),(1,1),(2,1),(3,1),(4,1)} {\node at \x [vertex] {};}
\draw (0,1)--(4,1);
\draw (0,1) to [out=30, in =150] (4,1);
\draw (1,1) to [out=30, in =150] (4,1);
\node at (-1, 1) {$G$};
\end{tikzpicture}
\end{tabular}
\caption{A graph $G$ and its subgraphs $H, K$ such that $I(G)= I(H) \cup I(K)$.}
\end{figure}
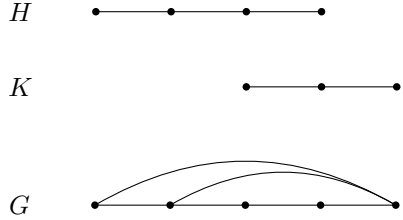
\begin{example}
For a graph $G$ and a vertex $v \in V(G)$, consider two subgraphs $G \setminus \{v\}$ and $G \setminus N_G (v)$ of $G$. We have
\begin{align*}
&(V(G) \setminus \{v\}) \setminus (V(G) \setminus N_G (v)) = N_G (v) ,\\
&(V(G) \setminus N_G (v)) \setminus (V(G) \setminus \{v\}) = \{v \}, \\
&(G \setminus \{v\}) \cap (G \setminus N_G (v)) = G \setminus \neib{G}{v} .
\end{align*}
Then, by Lemma \ref{ind pushout}, we have
\begin{align*}
I(G) = I(G \setminus \{v\}) \cup_{I(G \setminus \neib{G}{v})} I(G \setminus N_G (v)).
\end{align*}
Since $I(G \setminus N_G (v)) = I(G \setminus \neib{G}{v}) * \{v\} $, we obtain a cofiber sequence 
\begin{align*}
\xymatrix{
I(G \setminus \neib{G}{v}) \ar@{^{(}->}[r] & I(G \setminus \{v\}) \ar[r] & I(G),
}
\end{align*}
which was studied by Adamaszek \cite[Proposition 3.1]{Adamaszek12}.
\end{example}

\begin{lemma}
\label{mapping cylinder}
Let $X$ be a CW complex and $X_1, X_2$ be subcomplexes of $X$ such that $X=X_1 \cup X_2$. If the inclusion maps $i_1: X_1\cap X_2 \to X_1$ and $i_2 : X_1 \cap X_2 \to X_2$ are null-homotopic, then we have
\begin{align*}
X \simeq X_1 \vee X_2 \vee \Sigma (X_1 \cap X_2) .
\end{align*}
\end{lemma}
\begin{proof}
Consider the mapping cylinder $M(i_1, i_2)$ of $i_1, i_2$. Let $u \in X_1$ and $v \in X_2$ be points such that $i_1 \simeq c_u$ and $i_2 \simeq c_v$, where $c_u : X_1 \cap X_2 \to X_1$ and $c_v :X_1 \cap  X_2 \to X_2$ are the constant map to $u$ and $v$, respectively. Then, we have
\begin{align*}
X = X_1 \cup X_2 \simeq M(i_1, i_2) \simeq M(c_u, c_v) = X_1 \vee_u \Sigma(X_1 \cap X_2) \vee_v X_2.
\end{align*}
This is the desired conclusion.
\end{proof}

\begin{proof}[Proof of Theorem \ref{splitting}]
Consider two full subgraphs $K_1, K_2$ of $\lex{G}{H}$ defined by
\begin{align*}
&K_1=\lex{(G \setminus N_G (v))}{H} ,\\
&K_2=\lex{(G \setminus \{v\})}{H} .
\end{align*}\
Then we have
\begin{align*}
&V(K_1) \setminus V(K_2) = \{v\} \times V(H) ,\\
&V(K_2) \setminus V(K_1) = N_G (v) \times V(H) , \\
&K_1 \cap K_2 =\lex{(G \setminus \neib{G}{v})}{H}.
\end{align*}
It follows that $v_1 v_2 \in E(\lex{G}{H})$ for any vertices $v_1 \in V(K_1) \setminus V(K_2)$ and $v_2 \in V(K_2) \setminus V(K_1)$ since $u v \in E(G)$ for any $u \in N_G (v)$. So, by Lemma \ref{ind pushout}, we obtain
\begin{align*}
I(\lex{G}{H}) = I(\lex{(G \setminus N_G (v))}{H}) \cup_{I(\lex{(G \setminus \neib{G}{v})}{H})} I(\lex{(G \setminus \{v\})}{H})  .
\end{align*} 

If $G \setminus \neib{G}{v} = \emptyset$, then 
\begin{align*}
I(\lex{(G \setminus \neib{G}{v})}{H}) &= I(\lex{\emptyset}{H}) = I(\emptyset) = \emptyset, \\
I(\lex{(G \setminus N_G (v))}{H}) &= I(\lex{\{v\}}{H}) = I(H).
\end{align*}
So, the desired formula is obtained directly.

Suppose that $G \setminus \neib{G}{v} \neq \emptyset$. Let $i : I(\lex{(G \setminus \neib{G}{v})}{H}) \to I(\lex{(G \setminus N_G (v))}{H})$ and 
$j: I(\lex{(G \setminus \neib{G}{v})}{H}) \to I(\lex{(G \setminus \{v\})}{H})$ be the inclusion maps. By Proposition \ref{disjoint union and join}, we have
\begin{align*}
I(\lex{(G \setminus N_G (v))}{H}) &= I(\lex{((G \setminus \neib{G}{v})  \sqcup \{v\})}{H}) \\
&=  I(\lex{(G \setminus \neib{G}{v})}{H}) * I(H), \\
I(\lex{(G \setminus \{v\})}{H}) &= I(\lex{((G \setminus \{v, w\}) \sqcup \{w\})}{H}) \\
&= I(\lex{(G \setminus \{v, w\})}{H}) * I(H).
\end{align*}
The third equality follows from $N_G (w) = \{v\}$. 
Here, $I(H)$ is non-empty since $H$ is non-empty. Let $x \in I(H)$ be a point. Then, we have
\begin{align*}
I(\lex{(G \setminus \neib{G}{v})}{H}) * \{x\} &\subset  I(\lex{(G \setminus \neib{G}{v})}{H}) * I(H), \\
I(\lex{(G \setminus \neib{G}{v})}{H}) * \{x\} &\subset I(\lex{(G \setminus \{v, w\})}{H}) * I(H) .
\end{align*}
The second inclusion follows from $\{v , w\} \subset \neib{G}{v}$.
These inclusions indicate that $i, j$ are null-homotopic. Therefore, by Lemma \ref{mapping cylinder}, we obtain
\begin{align*}
I(\lex{G}{H}) = & I(\lex{(G \setminus N_G (v))}{H}) \cup_{I(\lex{(G \setminus \neib{G}{v})}{H})}  I(\lex{(G \setminus \{v\})}{H}) \\
\simeq &\Sigma I(\lex{(G \setminus \neib{G}{v} )}{H}) \vee \left(I(\lex{(G \setminus \neib{G}{v} )}{H}) * I(H) \right) \\
&\ \vee \left(I(\lex{(G \setminus\{v, w\})}{H}) * I(H) \right) .
\end{align*} 
So, the proof is completed.
\end{proof}

In order to derive Theorem \ref{forest} from Theorem \ref{splitting}, we need some topological observations, which we state in the following two lemmas.
\begin{lemma}
\label{disjoint suspension}
Let $X, Y$ be CW complexes. Then we have
\begin{align*}
\Sigma(X \sqcup Y) \simeq \Sigma X \vee \Sigma Y \vee \sphere{1}.
\end{align*}
\end{lemma}
\begin{proof}
Let $u, v$ be cone points of $\Sigma ( X \sqcup Y)$. Then we have
\begin{align*}
\Sigma(X \sqcup Y) = \Sigma X \cup_{\{u,v\}} \Sigma Y .
\end{align*}
For $x \in X$ and $y \in Y$, there are line segments $xu, xv \subset \Sigma X$ and $yu, yv \subset \Sigma Y$. So, the inclusion maps $\{u, v \} \to \Sigma X$, $\{u, v\} \to \Sigma Y$ are null-homotopic. Therefore, it follows from Lemma \ref{mapping cylinder} that
\begin{align*}
\Sigma (X \sqcup Y) &\simeq \Sigma X \vee \Sigma Y \vee \Sigma\{u, v\}  \\
&\simeq \Sigma X \vee \Sigma Y \vee \sphere{1}.
\end{align*}
\end{proof}

\begin{lemma}
\label{sphere join}
Let $A$, $B$, $C$ be CW complexes such that each of them is homotopy equivalent to a wedge sum of spheres. Then, both $A*B$ and $(A \sqcup B) *C$ are again homotopy equivalent to a wedge sum of spheres.
\end{lemma}
\begin{proof}
We first claim that for any CW complex $X, Y, Z$, we have
\begin{align*}
(X \vee Y) * Z \simeq (X * Z) \vee (Y * Z).
\end{align*}
This is because $X * Y$ is homotopy equivalent to $\Sigma( X \land Y)$ for any pointed CW complexes $(X, x_0)$ and $(Y,y_0)$. This homotopy equivalence yields
\begin{align*}
(X \vee Y) * Z  &\simeq \Sigma((X \vee Y) \land Z) \simeq \Sigma((X \land Z) \vee (Y \land Z)) \\
&\simeq \Sigma(X \land Z) \vee \Sigma(Y \land Z) \simeq (X * Z) \vee (Y * Z)
\end{align*}
as desired.

Let $A= \bigvee_i \sphere{a_i}$, $B= \bigvee_j \sphere{b_j}$, $C= \bigvee_k \sphere{c_k}$ be arbitrary wedge sums of spheres.
It follows from Lemma \ref{disjoint suspension} and above claim that
\begin{align*}
A * B &\simeq \left(\bigvee_i \sphere{a_i} \right) * \left(\bigvee_j \sphere{b_j} \right) \simeq \bigvee_i \left(\sphere{a_i} * \left( \bigvee_j \sphere{b_j} \right) \right) \\
&\simeq \bigvee_{i,j} \left( \sphere{a_i} * \sphere{b_j} \right) \simeq \bigvee_{i,j}\sphere{a_i + b_j +1},
\end{align*}
\begin{align*}
(A \sqcup B ) *C 
&\simeq \left( \left(\bigvee_i \sphere{a_i} \right) \sqcup \left(\bigvee_j \sphere{b_j} \right) \right) * \left(\bigvee_k \sphere{c_k} \right) \\
&\simeq \bigvee_k \left( \left( \left(\bigvee_i \sphere{a_i} \right) \sqcup \left(\bigvee_j \sphere{b_j} \right) \right)* \sphere{c_k} \right) \\
&\simeq \bigvee_k \left( \left(\left(\bigvee_i \sphere{a_i} \right) * \sphere{c_k} \right) \vee \left( \left(\bigvee_j \sphere{b_j} \right) * \sphere{c_k} \right) \vee \sphere{c_k +1} \right) \\
&\simeq \bigvee_k \left( \left(\bigvee_i \sphere{a_i + c_k +1} \right) \vee \left(\bigvee_j \sphere{b_j +c_k +1} \right)  \vee \sphere{c_k +1} \right) \\
&\simeq \left( \bigvee_{i,k} \sphere{a_i + c_k +1} \right) \vee \left( \bigvee_{j,k} \sphere{b_j + c_k +1} \right) \vee \left( \bigvee_k \sphere{c_k +1} \right).
\end{align*}
Therefore, we obtain the desired conclusion.
\end{proof} 

We are now ready to prove Theorem \ref{forest}.
\begin{proof}[Proof of Theorem \ref{forest}]
We prove the theorem by induction on $|V(G)|$.  Before we start, we confirm two cases.

First, suppose that $G$ is a star on at least $2$ vertices, namely $|V(G)| \geq 2$ and there exists $v \in V(G)$ such that $G \setminus \neib{G}{v} = \emptyset$. We have $u_1 u_2 \notin E(G)$ for any $u_1, u_2 \in N_G (v) = G \setminus \{v\}$ since $G$ is a forest. So, by Theorem \ref{splitting}, we get
\begin{align*}
I(\lex{G}{H}) & =  I(H)  \sqcup  I(\lex{(G \setminus \{v\})}{H}) \\
&=I(H) \sqcup \left(\mathop{*}_{|V(G)| - 1} I(H) \right) .
\end{align*}
Since $|V(G)|-1 \geq 1$, the join of copies of $I(H)$ is homotopy equivalent to a wedge sum of spheres by Lemma \ref{sphere join}. Therefore, $I(\lex{G}{H})$ is homotopy equivalent to a disjoint union of two wedge sums of spheres. 

Next, suppose that $G$ has no edges. Then $I(\lex{G}{H})$ is the join of $|V(G)|$ copies of $I(H)$, which is a wedge sum of spheres by Lemma \ref{sphere join}. 

Now we start the induction. 
The forest $G$ with $|V(G)| \leq 2$ is isomorphic to one of $L_1$, $L_2$ and $L_1 \sqcup L_1$. They are included in the above cases. Hence, for a forest $G$ with $|V(G)| \leq 2$, $I(\lex{G}{H})$ is homotopy equivalent to a wedge sum of spheres or a disjoint union of two wedge sums of spheres. 

Assume that for any forest $G'$ such that $|V(G')| \leq n$, $I(\lex{G'}{H})$ is homotopy equivalent to a wedge sum of spheres or a disjoint union of two wedge sums of spheres. Let $G$ be a forest with at least one edge such that $|V(G)|=n+1$ and $G \setminus \neib{G}{v} \neq \emptyset$ for any $v \in V(G)$. Then, since $G$ is a forest, there exists $w \in V(G)$ such that $N_G (w) = \{v\}$ for some $v \in V(G)$ (namely a leaf $w$ of $G$). We write $G_1 = G \setminus \neib{G}{v}$ and $G_2 =G \setminus\{v, w\}$. Then, $G_1, G_2$ are forests such that $|V(G_1)| \leq n-1$, $|V(G_2)| \leq n-1$. 

Since $G_1=G \setminus \neib{G}{v}$ is not empty, it follows from Theorem \ref{splitting} that
\begin{align*}
I(\lex{G}{H}) \simeq &\Sigma I(\lex{G_1}{H}) \vee \left(I(\lex{G_1}{H}) * I(H) \right) \vee \left(I(\lex{G_2}{H}) * I(H) \right) .
\end{align*}
By the assumption of the induction, $I(\lex{G_1}{H})$ and $I(\lex{G_2}{H})$ are homotopy equivalent to a wedge sum of spheres or a disjoint union of two wedge sums of spheres. Therefore, by Lemma \ref{sphere join}, $I(\lex{G}{H})$ is homotopy equivalent to a wedge sum of spheres.
So, the proof is completed.
\end{proof}
\begin{remark}
\label{contractible}
For a graph $H$, suppose that $I(H)$ is contractible. Then, for a forest $G$, we have $I(\lex{G}{H}) \simeq I(G)$. We can prove this fact in the same way as in the proof of Theorem \ref{forest}. 
\end{remark}
\begin{example}
Recall that a graph $G$ is {\it chordal} if it contains no cycle of length at least $4$. Kawamura \cite[Theorem 1.1]{Kawamura10} proved that the independence complex of a chordal graph is either contractible or homotopy equivalent to a wedge sum of spheres. In particular, Ehrenborg and Hetyei \cite[Corollary 6.1]{EhrenborgHetyei06} proved that the independence complex of a forest is either contractible or homotopy equivalent to a single sphere. So, it follows from Theorem \ref{forest} and Remark \ref{contractible} that $I(\lex{G}{H})$ is either contractible or homotopy equivalent to a wedge sum of spheres if $G$ is a forest and $H$ is a chordal graph.
\end{example} 

\section{Explicit Calculations}
\label{explicit calculations}
In this section, we offer two examples of explicit calculations on $I(\lex{G}{H})$. First, we prove Theorem \ref{line theorem}.
\begin{proof}[Proof of Theorem \ref{line theorem}]
For $m=1,2,3$, it follows from Proposition \ref{disjoint union and join} that 
\begin{align*}
I(\lex{L_1}{H}) &= I(H) \simeq {\bigvee}_n \sphere{k} , \\
I(\lex{L_2}{H}) &= I(H) \sqcup I(H) \simeq \left( {\bigvee}_n \sphere{k} \right) \sqcup \left( {\bigvee}_n \sphere{k} \right), \\
I(\lex{L_3}{H}) &= I(H) \sqcup (I(H) * I(H)) \\
&\simeq \left( {\bigvee}_n \sphere{k} \right) \sqcup \left( \left( {\bigvee}_n \sphere{k} \right) * \left( {\bigvee}_n \sphere{k} \right) \right) \\
&\simeq \left( {\bigvee}_n \sphere{k} \right) \sqcup \left(  {\bigvee}_n \left( \sphere{k} * \left( {\bigvee}_n \sphere{k} \right) \right) \right) \\
&\simeq \left( {\bigvee}_n \sphere{k} \right) \sqcup \left(  {\bigvee}_n \left( {\bigvee}_n \sphere{k} * \sphere{k} \right) \right) \\
&\simeq \left( {\bigvee}_n \sphere{k} \right) \sqcup  \left( {\bigvee}_{n^2} \sphere{2k+1} \right).
\end{align*}
For $r \geq 1$, let $G=L_{r+3}$ and $v=r+2, w=r+3 \in V(L_{r+3})$. Then we have $N_G (w)=\{v\}$ and $G \setminus \neib{G}{v} = L_r \neq \emptyset$. So, by Theorem \ref{splitting}, we obtain
\begin{align}
&I(\lex{L_{r+3}}{H}) \nonumber \\
\simeq &\Sigma I(\lex{L_r}{H}) \vee \left(I(\lex{L_r}{H}) * I(H) \right) \ \vee \left(I(\lex{L_{r+1}}{H}) * I(H) \right) \nonumber \\
\simeq &\Sigma I(\lex{L_r}{H}) \vee \left(I(\lex{L_r}{H}) * \left( {\bigvee}_n \sphere{k} \right) \right) \ \vee \left(I(\lex{L_{r+1}}{H}) *  \left( {\bigvee}_n \sphere{k} \right) \right)  \nonumber \\
\simeq &\Sigma I(\lex{L_r}{H}) \vee  \left( {\bigvee}_n I(\lex{L_r}{H}) * \sphere{k} \right) \vee \left( {\bigvee}_n I(\lex{L_{r+1}}{H}) * \sphere{k} \right) \nonumber \\
\simeq &\Sigma I(\lex{L_r}{H}) \vee  \left( {\bigvee}_n \Sigma^{k+1} I(\lex{L_r}{H}) \right) \vee \left( {\bigvee}_n \Sigma^{k+1} I(\lex{L_{r+1}}{H}) \right) . \label{Lm recursive}
\end{align}

Define a CW complex $X_{m,n,k}$ for $m\geq 1$, $n \geq 1$ and $k \geq 0$ by 
\begin{align*}
X_{m,n,k}= \bigvee_{d \geq 0} \left( \bigvee_{p \geq 0} \left( {\bigvee}_{N_{m,n,k}(p,d)} \sphere{d} \right) \right) ,
\end{align*}
where
\begin{align*}
N_{m,n,k}(p,d) &= n^p \binom{d-pk+1}{p} \binom{p+1}{3(d-pk+1)-m} .
\end{align*}
We note that $N_{m,n,k}(p,d) >0$ for non-negative integers $p, d$ if and only if $d-pk+1 \geq p$ and $p+1 \geq 3(d-pk+1)-m \geq 0 $, namely
\begin{align*}
pk-1 +\max \left\{p, \frac{m}{3} \right\} \leq d \leq pk+\frac{m+p-2}{3} .
\end{align*}
The above inequality implies that $p \leq \frac{m+1}{2}$. So, it follows that
\begin{align*}
X_{m,n,k}= 
\bigvee_{0 \leq p \leq \frac{m+1}{2}} \left( \bigvee_{pk -1 +\max \left\{p, \frac{m}{3} \right\} \leq d \leq pk+\frac{m+p-2}{3}}
\left( {\bigvee}_{N_{m,n,k}(p,d)} \sphere{d} \right) \right) .
\end{align*}

In order to complete the proof, it is sufficient to show that $I(\lex{L_m}{H}) \simeq X_{m,n,k}$ for $m \geq 4$. First, the explicit descriptions of $X_{1,n,k}$, $X_{2,n,k}$ and $X_{3,n,k}$ are obtained as follows.
\begin{align*}
X_{1,n,k} = &\bigvee_{0 \leq p \leq 1} \left( \bigvee_{pk-1+ \max \left\{ p, \frac{1}{3} \right\} \leq d \leq pk+\frac{1+p-2}{3}}
\left( {\bigvee}_{N_{1,n,k}(p,d)} \sphere{d} \right) \right) \\
= &\bigvee_{p=0,1} \left( \bigvee_{pk-1+ \max \left\{ p, \frac{1}{3} \right\}  \leq d \leq pk+\frac{p-1}{3}}
\left( {\bigvee}_{N_{1,n,k}(p,d)} \sphere{d} \right) \right) \\
= &\left( \bigvee_{-\frac{2}{3} \leq d \leq -\frac{1}{3}} \left( {\bigvee}_{N_{1,n,k}(0,d)} \sphere{d} \right) \right)
\vee \left( \bigvee_{k \leq d \leq k} \left( {\bigvee}_{N_{1,n,k}(1,d)} \sphere{d} \right) \right) \\
= & {\bigvee}_{N_{1,n,k}(1,k)} \sphere{k}  \\
= & {\bigvee}_{n^1 \binom{1}{1} \binom{2}{2}} \sphere{k} \\
= & {\bigvee}_n \sphere{k} .
\end{align*}
\begin{align*}
X_{2,n,k}
= &\bigvee_{0 \leq p \leq \frac{3}{2}} \left( \bigvee_{pk-1+ \max \left\{ p, \frac{2}{3} \right\} \leq d \leq pk+\frac{2+p-2}{3}}
\left( {\bigvee}_{N_{2,n,k}(p,d)} \sphere{d} \right) \right) \\
= &\bigvee_{p=0,1} \left( \bigvee_{pk-1+ \max \left\{ p, \frac{2}{3} \right\}  \leq d \leq pk+\frac{p}{3}}
\left( {\bigvee}_{N_{2,n,k}(p,d)} \sphere{d} \right) \right) \\
=&\left( \bigvee_{-\frac{1}{3} \leq d \leq 0} \left( {\bigvee}_{N_{2,n,k}(0,d)} \sphere{d} \right) \right)
\vee \left( \bigvee_{k \leq d \leq k+\frac{1}{3}} \left( {\bigvee}_{N_{2,n,k}(1,d)} \sphere{d} \right) \right) \\
= &\left( {\bigvee}_{N_{2,n,k}(0,0)} \sphere{0} \right) \vee \left( {\bigvee}_{N_{2,n,k}(1,k)} \sphere{k} \right) \\
=&\left( {\bigvee}_{n^0 \binom{1}{0} \binom{1}{1}} \sphere{0} \right) \vee \left( {\bigvee}_{n^1 \binom{1}{1} \binom{2}{1}} \sphere{k} \right) \\
=&\sphere{0} \vee \left( {\bigvee}_{2n} \sphere{k} \right).
\end{align*}
\begin{align*}
X_{3,n,k}
= &\bigvee_{0 \leq p \leq 2} \left( \bigvee_{pk-1+ \max \left\{ p, 1 \right\} \leq d \leq pk+\frac{3+p-2}{3}}
\left( {\bigvee}_{N_{3,n,k}(p,d)} \sphere{d} \right) \right) \\
=&\bigvee_{p=0,1,2} \left( \bigvee_{pk-1+ \max \left\{ p, 1 \right\} \leq d \leq pk+\frac{p+1}{3}}
\left( {\bigvee}_{N_{3,n,k}(p,d)} \sphere{d} \right) \right) \\
= &\left( \bigvee_{0 \leq d \leq \frac{1}{3}} \left( {\bigvee}_{N_{3,n,k}(0,d)} \sphere{d} \right) \right)
\vee \left( \bigvee_{k\leq d \leq k+\frac{2}{3}} \left( {\bigvee}_{N_{3,n,k}(1,d)} \sphere{d} \right) \right) \\
&\ \vee \left( \bigvee_{2k+1 \leq d \leq 2k+1} \left( {\bigvee}_{N_{3,n,k}(2,d)} \sphere{d} \right) \right) \\
= &\left( {\bigvee}_{N_{3,n,k}(0,0)} \sphere{0} \right) \vee \left( {\bigvee}_{N_{3,n,k}(1,k)} \sphere{k} \right) \\
&\ \vee \left( {\bigvee}_{N_{3,n,k}(2,2k+1)} \sphere{2k+1} \right) \\
= &\left( {\bigvee}_{n^0 \binom{1}{0} \binom{1}{0}} \sphere{0} \right) \vee \left( {\bigvee}_{n^1 \binom{1}{1} \binom{2}{0}} \sphere{k} \right)
 \vee \left( {\bigvee}_{n^2 \binom{2}{2} \binom{3}{3}} \sphere{2k+1} \right) \\
= &\sphere{0} \vee \left( {\bigvee}_n \sphere{k} \right) \vee \left( {\bigvee}_{n^2} \sphere{2k+1} \right).
\end{align*}
We next show that 
\begin{align}
\label{X recursive}
X_{m+3,n,k} = \Sigma X_{m,n,k} \vee \left( {\bigvee}_n \Sigma^{k+1} X_{m,n,k} \right) \vee \left( {\bigvee}_n \Sigma^{k+1} X_{m+1,n,k} \right).
\end{align}
We have
\begin{align*}
&\sum_{p \geq 0} \left(N_{m,n,k}(p,d-1) + n \cdot N_{m,n,k}(p,d-k-1) +n \cdot N_{m+1,n,k}(p,d-k-1) \right) \\
=&\sum_{p \geq 0} \left( n^p \binom{(d-1)-pk+1}{p} \binom{p+1}{3((d-1)-pk+1)-m} \right. \\
&\ + n^{p+1} \binom{(d-k-1)-pk+1}{p} \binom{p+1}{3((d-k-1)-pk+1)-m} \\
&\ \left. + n^{p+1} \binom{(d-k-1)-pk+1}{p} \binom{p+1}{3((d-k-1)-pk+1)-(m+1)} \right) \\
=&\sum_{p \geq 0} \left( n^p \binom{d-pk}{p} \binom{p+1}{3(d-pk)-m} \right. \\
&\ +n^{p+1}\binom{d-(p+1)k}{p} \binom{p+1}{3(d-(p+1)k)-m} \\
&\ \left. +n^{p+1} \binom{d-(p+1)k}{p} \binom{p+1}{3(d-(p+1)k)-(m+1)} \right) \\
=&\sum_{p \geq 0} n^p \binom{d-pk}{p} \binom{p+1}{3(d-pk)-m} \\
&\ +\sum_{p \geq 0} n^{p+1} \binom{d-(p+1)k}{p} \binom{p+2}{3(d-(p+1)k)-(m+1)} \\
=&\sum_{p \geq 0 } n^p \binom{d-pk}{p} \binom{p+1}{3(d-pk)-m} \\
&\ +\sum_{q=p+1 \geq 1 } n^q \binom{d-qk}{q-1} \binom{q+1}{3(d-qk)-m} \\
=&\sum_{p \geq 0} n^p \binom{d-pk+1}{p} \binom{p+1}{3(d-pk)-m} \\
=&\sum_{p \geq 0} N_{m+3,k}(p,d) .
\end{align*}
So, we conclude that
\begin{align*}
&\Sigma X_{m,n,k} \vee \left( {\bigvee}_n \Sigma^{k+1} X_{m,n,k} \right) \vee \left( {\bigvee}_n \Sigma^{k+1} X_{m+1,n,k} \right)\\
= &\bigvee_{d \geq 0} \left( \bigvee_{p \geq 0} \left( {\bigvee}_{N_{m,n,k}(p,d-1) + n \cdot N_{m,n,k}(p,d-k-1) +n \cdot N_{m+1,n,k}(p,d-k-1)} \sphere{d} \right) \right) \\
= &\bigvee_{d \geq 0} \left( {\bigvee}_{\sum_{p \geq 0} \left(N_{m,n,k}(p,d-1) + n \cdot N_{m,n,k}(p,d-k-1) +n \cdot N_{m+1,n,k}(p,d-k-1) \right)  } \sphere{d} \right) \\
= &\bigvee_{d \geq 0} \left( {\bigvee}_{\sum_{p \geq 0} N_{m+3,n,k}(p,d)} \sphere{d} \right) \\
= &\bigvee_{d \geq 0} \left( \bigvee_{p \geq 0} \left( {\bigvee}_{N_{m+3,n,k}(p,d)} \sphere{d} \right) \right) \\
=&X_{m+3,n,k} 
\end{align*}
as desired.

Now, we are ready to finish the proof by induction on $m$. By Lemma \ref{disjoint suspension}, we obtain 
\begin{align*}
\Sigma I(\lex{L_2}{H}) &\simeq \Sigma \left(\left( {\bigvee}_n \sphere{k} \right) \sqcup \left( {\bigvee}_n \sphere{k} \right) \right) \\
&\simeq \sphere{1} \vee \Sigma \left({\bigvee}_n \sphere{k} \right) \vee \Sigma \left( {\bigvee}_n \sphere{k} \right) \\
&\simeq \sphere{1} \vee \left({\bigvee}_n \sphere{k+1} \right) \vee \left( {\bigvee}_n \sphere{k+1} \right) \\
&=\sphere{1} \vee \left({\bigvee}_{2n} \sphere{k+1} \right), \\
\Sigma I(\lex{L_3}{H}) &\simeq \Sigma \left(\left( {\bigvee}_n \sphere{k} \right) \sqcup  \left( {\bigvee}_{n^2} \sphere{2k+1} \right) \right) \\
&\simeq \sphere{1} \vee \Sigma \left({\bigvee}_n \sphere{k} \right) \vee \Sigma \left( {\bigvee}_{n^2} \sphere{2k+1} \right) \\
&\simeq \sphere{1} \vee \left({\bigvee}_n \sphere{k+1} \right) \vee \left( {\bigvee}_{n^2} \sphere{2k+2} \right) .
\end{align*}
So, it follows that
\begin{align*}
\Sigma I(\lex{L_m}{H}) \simeq \Sigma X_{m,n,k}
\end{align*}
for $m =1,2,3$. Assume that $\Sigma I(\lex{L_r}{H}) \simeq \Sigma X_{r,n,k}$ and $\Sigma I(\lex{L_{r+1}}{H}) \simeq \Sigma X_{r+1,n,k}$ for some $r \geq 1$. By recursive relations (\ref{Lm recursive}) and (\ref{X recursive}), we have
\begin{align*}
&I(\lex{L_{r+3}}{H}) \\
\simeq &\Sigma I(\lex{L_r}{H}) \vee  \left( {\bigvee}_n \Sigma^{k+1} I(\lex{L_r}{H}) \right) \vee \left( {\bigvee}_n \Sigma^{k+1} I(\lex{L_{r+1}}{H}) \right) \\
\simeq &\Sigma X_{r,n,k} \vee \left( {\bigvee}_n \Sigma^{k+1} X_{r,n,k} \right) \vee \left( {\bigvee}_n \Sigma^{k+1} X_{r+1,n,k} \right) \\
=&X_{r+3,n,k}.
\end{align*}
Therefore, we obtain that $I(\lex{L_m}{H}) \simeq X_{m,n,k}$ for any $m \geq 4$ by induction. This is the desired conclusion.
\end{proof}

\begin{example}
Kozlov \cite[Proposition 5.2]{Kozlov99} proved that
\begin{align*}
I(C_n) &\simeq \left\{
\begin{aligned}
&\sphere{k - 1} \vee \sphere{k - 1}  & &(n =3k), \\
&\sphere{k-1} & &(n =3k+1), \\
&\sphere{k} & &(n =3k+2)  .
\end{aligned} \right. 
\end{align*}
Therefore, we can determine the homotopy types of $I(\lex{L_m}{C_n})$ for any $m \geq 1$ and $n \geq 3$ by Theorem \ref{line theorem}.
\end{example}

Recall that the homological connectivity of a space $X$, denoted by $\mathrm{conn}_H(X)$, is defined by
\begin{align*}
\mathrm{conn}_H(X)= \left\{
\begin{aligned}
&-2 & &(X = \emptyset), \\
&k & &(\widetilde{H}_i (X)=0 \text{ for any $i \leq k$, } \widetilde{H}_{k+1} (X) \neq 0 ),  \\
&\infty & &(\widetilde{H}_i (X) = 0 \text{ for any $i$ }),
\end{aligned} \right.
\end{align*}
where $\widetilde{H}_i (X)$ is the reduced $i$th homology group of $X$.
Though Theorem \ref{line theorem} completely determines the homotopy type of $I(\lex{L_m}{H})$ with $I(H) \simeq {\bigvee}_n \sphere{k}$, it is hard to obtain the homological connectivity of $I(\lex{L_m}{H})$ immediately from Theorem \ref{line theorem}. Here we compute the homological connectivity of $I(\lex{L_m}{H})$ as a corollary.  
\begin{corollary}
\label{line corollary}
Let $H$ be a graph such that $I(H) \simeq {\bigvee}_n \sphere{k}$ with $n \geq 1$, $k \geq 0$. Then we have
\begin{align*}
\mathrm{conn}_H(I(\lex{L_{3l+i}}{H})) = \left\{
\begin{aligned}
&l-2 & &(i=0), \\
&k+l-1 & &(i=1), \\
&l -1& &(i=2).
\end{aligned} \right.
\end{align*}
\end{corollary}
\begin{proof}
Recall from the proof of Theorem \ref{line theorem} that there is a recursive relation
\begin{align*}
&I(\lex{L_{m+3}}{H}) \\
\simeq &\Sigma I(\lex{L_m}{H}) \vee  \left( {\bigvee}_n \Sigma^{k+1} I(\lex{L_m}{H}) \right) \vee \left( {\bigvee}_n \Sigma^{k+1} I(\lex{L_{m+1}}{H}) \right).
\end{align*}
So, we obtain
\begin{align*}
&\mathrm{conn}_H (I(\lex{L_{m+3}}{H})) \\
= &\min \left\{ \mathrm{conn}_H(\Sigma  I(\lex{L_m}{H})), \mathrm{conn}_H(\Sigma^{k+1} I(\lex{L_{m+1}}{H})) \right\} .
\end{align*}
The base cases are
\begin{align*}
\mathrm{conn}_H (I(\lex{L_1}{H})) &= \mathrm{conn}_H \left({\bigvee}_n \sphere{k} \right) =k-1, \\
\mathrm{conn}_H (I(\lex{L_2}{H})) &= \mathrm{conn}_H \left(\left( {\bigvee}_n \sphere{k} \right) \sqcup \left( {\bigvee}_n \sphere{k} \right) \right) = -1, \\
\mathrm{conn}_H (I(\lex{L_3}{H})) &= \mathrm{conn}_H \left(\left( {\bigvee}_n \sphere{k} \right) \sqcup  \left( {\bigvee}_{n^2} \sphere{2k+1} \right) \right) = -1. 
\end{align*}
Therefore, we can prove the corollary by induction.
\end{proof}

We move on to the second example.
We denote the complete graph on $n$ vertices by $K_n$. For $n \geq 2$, it is obvious that 
\begin{align*}
I(K_n) = {\bigvee}_{n-1} \sphere{0}.
\end{align*}
As the second example in this section, we show that the homological connectivity of $I(\lex{G}{K_n})$ for any forest $G$ is determined by the {\it independent domination number} of $G$ when $n \geq 2$. For a graph $G$ and a subset $S \subset V(G)$, $S$ is a {\it dominating set} of $G$ if $V(G) = \bigcup_{u \in S} \neib{G}{u}$.
The domination number $\gamma (G)$ of $G$ is the minimum cardinality of a dominating set of $G$. The relationship between the domination number of $G$ and the homological connectivity of $I(G)$ was argued by Meshulam \cite{Meshulam03}, who proved that for a chordal graph $G$, $i < \gamma(G)$ implies $\widetilde{H}_{i-1} (I(G)) =0$  (\cite[Theorem 1.2 (iii)]{Meshulam03}). This is equivalent to state that $\mathrm{conn}_H (I(G)) \geq \gamma(G) -2$. 
This theorem can be used to deduce a result of Aharoni, Berger and Ziv \cite{AharoniBergerZiv02}.

A dominating set $S$ of $G$ is called {\it an independent dominating set} if $S$ is an independent set. The independent domination number $i (G)$ is the minimum cardinality of an independent dominating set of $G$. It is obvious that $i(G) \geq \gamma(G)$ since an independent dominating set is a dominating set. 
\begin{theorem}
\label{connectivity and domination}
Let $G$ be a forest. Then, for any $n \geq 2$, we have
\begin{align}
\label{domination}
\mathrm{conn}_H (I(\lex{G}{K_n})) = i (G) -2.
\end{align}
\end{theorem}
\begin{proof}
We first consider two cases.
\begin{itemize}
\item If $G \setminus \neib{G}{v} = \emptyset$ for some $v \in V(G)$, then we have $i(G) = 1$ and 
\begin{align*}
\mathrm{conn}_H (I(\lex{G}{K_n})) &= \mathrm{conn}_H \left(  \left({\bigvee}_{n-1} \sphere{0} \right) \sqcup \left( {\bigvee}_{(n-1)^{|V(G)| -1} } \sphere{|V(G)|-2} \right) \right) \\
&=-1
\end{align*} 
by Theorem \ref{splitting}.
\item If $G$ has no edges, then we have $i (G) = |V(G)|$ and 
\begin{align*}
\mathrm{conn}_H (I(\lex{G}{K_n})) &=\mathrm{conn}_H \left(  {\bigvee}_{(n-1)^{|V(G)| } } \sphere{|V(G)|-1} \right) \\
&=|V(G)|-2.
\end{align*}
\end{itemize}
Therefore, equation (\ref{domination}) holds in these two cases. 

We prove the theorem by induction on $|V(G)|$. Since $L_1$, $L_2$ and $L_1 \sqcup L_1$ are included in the above two cases, equation (\ref{domination}) holds for $G$ such that $|V(G)| \leq 2$.
Assume that (\ref{domination}) holds for any forest $G'$ such that $|V(G')| \leq r$ with $r \geq 2$. Let $G$ be a forest such that $|V(G)|=r+1$ and there exists $v, w \in V(G)$ such that $N_G (w) = \{v\}$ and $G \setminus \neib{G}{v} \neq \emptyset$.
By Theorem \ref{splitting}, we obtain
\begin{align*}
I(\lex{G}{K_n}) \simeq &\Sigma I(\lex{(G \setminus \neib{G}{v} )}{K_n}) \vee \left(I(\lex{(G \setminus \neib{G}{v} )}{K_n}) * \left( {\bigvee}_{n-1} \sphere{0} \right) \right) \\
&\ \vee \left(I(\lex{(G \setminus\{v, w\})}{K_n}) * \left( {\bigvee}_{n-1} \sphere{0} \right) \right)\\
\simeq &\Sigma I(\lex{(G \setminus \neib{G}{v} )}{K_n}) \vee \left({\bigvee}_{n-1} \Sigma I(\lex{(G \setminus \neib{G}{v} )}{K_n})  \right) \\
&\ \vee \left({\bigvee}_{n-1} \Sigma I(\lex{(G \setminus\{v, w\})}{K_n}) \right) \\
= &\left({\bigvee}_{n} \Sigma I(\lex{(G \setminus \neib{G}{v} )}{K_n})  \right) \vee \left({\bigvee}_{n-1} \Sigma I(\lex{(G \setminus\{v, w\})}{K_n}) \right).
\end{align*}
Hence, we get
\begin{align*}
&\mathrm{conn}_H (I(\lex{G}{K_n})) \\
= &\min \left\{ \mathrm{conn}_H (I(\lex{(G \setminus \neib{G}{v} )}{K_n})) +1,
\mathrm{conn}_H (I(\lex{(G \setminus\{v, w\})}{K_n})) +1 \right\} . 
\end{align*}
$G \setminus \neib{G}{v}$ and $G \setminus \{v, w\}$ are the forests which satisfy $|V(G \setminus \neib{G}{v})| \leq r-1$, $|V(G \setminus \{v, w\})| \leq r-1$. So, by the assumption of induction, we get
\begin{align*}
\mathrm{conn}_H (I(\lex{G}{K_n})) 
= &\min \left\{ i(G \setminus \neib{G}{v}) -1 , i(G \setminus \{v,w\}) -1 \right\}.
\end{align*}
Here, we have $i(G \setminus \neib{G}{v}) \geq i(G) -1$. It is because if there exists an independent dominating set $S$ of $G \setminus \neib{G}{v}$ with $|S| < i(G) - 1$, then $S \cup \{v\}$ is an independent dominating set of $G$ such that $|S \cup \{u\}| < i(G)$, a contradiction. For the same reason, we also have $i(G \setminus \{v, w \}) \geq i(G) -1$. 

An independent dominating set of $G$ must contain either $v$ or $w$ since $N_G (w) =\{v\}$. If there exists an independent dominating set $S$ of $G$ such that $|S| = i(G)$ and $v \in S$, then $S'=S \setminus \{v\}$ is an independent dominating set of $G \setminus \neib{G}{v}$ with $|S'|=i(G) -1$ since $S \cap \neib{G}{v} = \{v\}$. Thus, in this case, we obtain $i(G \setminus \neib{G}{v}) = i(G) -1$. If there exists an independent dominating set $S$ of $G$ such that $|S| = i(G)$ and $w \in S$, then $S'' = S \setminus \{w\}$ is an independent dominating set of $G \setminus \{v, w\}$ with $|S''|=i(G) -1$ since $v \notin S$. So, in this case, we get $i(G \setminus \{v, w \}) = i(G) -1$.

Above argument shows that 
\begin{align*}
\min \left\{ i(G \setminus \neib{G}{v}) -1 , i(G \setminus \{v,w\}) -1 \right\} = i(G) -2.
\end{align*}
Therefore, equation (\ref{domination}) holds for $G$. By induction, we get the desired conclusion.
\end{proof}


\begin{thebibliography}{10}

\bibitem{Adamaszek12}
Micha{\l} Adamaszek,
\newblock Splittings of independence complexes and the powers of cycles,
\newblock {\em Journal of Combinatorial Theory, Series A}, 119:1031--1047,
  2012.

\bibitem{AharoniBergerZiv02}
Ron Aharoni, Eli Berger, and Ran Ziv,
\newblock A tree version of {K}{\H{o}}nig's theorem,
\newblock {\em Combinatorica}, 22(3):335--343, 2002.

\bibitem{Barmak13}
Jonathan~Ariel Barmak,
\newblock Star clusters in independence complexes of graphs,
\newblock {\em Advances in Mathematics}, 241:33--57, 2013.

\bibitem{BousquetmelouLinussonNevo08}
Mireille Bousquet-M{\'{e}}lou, Svante Linusson, and Eran Nevo,
\newblock On the independence complex of square grids,
\newblock {\em Journal of Algebraic Combinatorics}, 27:423--450, 2008.

\bibitem{Chari00}
Manoj~K. Chari,
\newblock On discrete {M}orse functions and combinatorial decompositions,
\newblock {\em Discrete Mathematics}, 217:101--113, 2000.

\bibitem{EhrenborgHetyei06}
Richard Ehrenborg and G{\'{a}}bor Hetyei,
\newblock The topology of the independence complex,
\newblock {\em European Journal of Combinatorics}, 27:906--923, 2006.

\bibitem{Engstrom09}
Alexander Engstr{\"{o}}m,
\newblock Complexes of directed trees and independence complexes,
\newblock {\em Discrete Mathematics}, 309:3299--3309, 2009.

\bibitem{Forman98}
Robin Forman,
\newblock {M}orse theory for cell complexes,
\newblock {\em Advances in Mathematics}, 134:90--145, 1998.

\bibitem{GellerStahl75}
Dennis Geller and Saul Stahl,
\newblock The chromatic number and other functions of the lexicographic
  product,
\newblock {\em Journal of Combinatorial Theory, Series B}, 19:87--95, 1975.

\bibitem{Harary69}
Frank Harary,
\newblock {\em Graph Theory},
\newblock Addison-Wesley Publishing Company, 1969.

\bibitem{Iriye12}
Kouyemon Iriye,
\newblock On the homotopy types of the independence complexes of grid graphs
  with cylindrical identification,
\newblock {\em Kyoto Journal of Mathematics}, 52(3):479--501, 2012.

\bibitem{Kawamura10}
Kazuhiro Kawamura,
\newblock Independence complexes of chordal graphs,
\newblock {\em Discrete Mathematics}, 310:2204--2211, 2010.

\bibitem{Kozlov99}
Dmitry~N. Kozlov,
\newblock Complexes of directed trees,
\newblock {\em Journal of Combinatorial theory, Series A}, 88:112--122, 1999.

\bibitem{Meshulam03}
Roy Meshulam,
\newblock Domination numbers and homology,
\newblock {\em Journal of Combinatorial Theory, Series A}, 102:321--330, 2003.

\bibitem{Thapper08}
Johan Thapper,
\newblock Independence complexes of cylinders constructed from square and
  hexagonal grid graphs,
\newblock {\em arXiv e-prints}, arXiv:0812.1165, 2008.

\bibitem{VandermeulenVantuyl17}
Kevin~N. {Vander Meulen} and Adam {Van Tuyl},
\newblock Shellability, vertex decomposability, and lexicographical products of
  graphs,
\newblock {\em Contributions to Discrete Mathematics}, 12(2):63--68, 2017.

\end{thebibliography}
\end{document}